\newif \ifIMA
\ifIMA
\documentclass[final]{imanum}
\jno{\#}
\received{14 March 2017}
\else
\documentclass[11pt]{article}
\fi

\ifIMA
\else
\usepackage[a4paper]{geometry}
\fi
\usepackage{amsmath,amsfonts,amssymb,amsthm}
\usepackage{graphicx,verbatim,mathtools,color}
\usepackage[only,llparenthesis,rrparenthesis]{stmaryrd}
\usepackage{bm}

\definecolor{dblue}{rgb}{0.,0.,0.8}

\ifIMA
\hypersetup{bookmarks=true, colorlinks, linkcolor=black, citecolor=black, urlcolor=dblue, plainpages=true, pdfwindowui=false,pdfstartview={FitH}}
\else
\usepackage{hyperref}
\hypersetup{bookmarks=true, colorlinks, linkcolor=dblue, citecolor=dblue, urlcolor=dblue, plainpages=true, pdfwindowui=false,pdfstartview={FitH}}
\fi 

\numberwithin{equation}{section}

\ifIMA
\else
\theoremstyle{plain}
\newtheorem{theorem}{Theorem}[section]
\newtheorem{definition}[theorem]{Definition}
\newtheorem{lemma}[theorem]{Lemma}
\newtheorem{corollary}[theorem]{Corollary}
\theoremstyle{remark}
\newtheorem{remark}{Remark}[section]
\fi

\newcommand{\ccite}[1]{\ifIMA\citet*[]{#1}\else\cite{#1}\fi}
\newcommand{\ocite}[2][]{\ifIMA\citet*[#1]{#2}\else\cite[#1]{#2}\fi}
\newcommand{\curl}[1]{\ifIMA\href{#1}{#1}\else\url{#1}\fi}

\newcommand{\lld}{\llparenthesis}
\newcommand{\rrd}{\rrparenthesis}

\newcommand{\calB}{\mathcal{B}}
\newcommand{\calE}{\mathcal{E}}
\newcommand{\calI}{\mathcal{I}}
\ifIMA
\newcommand{\calQ}{\mathcal{P}}
\else
\newcommand{\calQ}{\mathcal{Q}}
\fi
\newcommand{\calP}{\mathcal{P}}
\newcommand{\calR}{\mathcal{R}}

\newcommand{\calT}{\mathcal{T}}
\newcommand{\calV}{\mathcal{V}}

\newcommand{\ver}{{\bm{a}}}

\newcommand{\dd}{\mathrm{d}}
\renewcommand{\dim}{d}
\newcommand{\abs}[1]{\lvert#1\rvert}
\newcommand{\Labs}[1]{\left\lvert#1\right\rvert} 
\newcommand{\tends}{\rightarrow}
\newcommand{\norm}[1]{\lVert#1\rVert}
\newcommand{\p}{\partial}
\renewcommand{\th}{{h\tau}}

\newcommand{\eval}[2]{\left. #1\right|_{#2}}
\DeclareMathOperator{\Div}{div}

\DeclareMathOperator{\diam}{diam}
\DeclareMathOperator*{\argmin}{argmin}

\newcommand{\pair}[2]{\langle #1,#2 \rangle}
\newcommand{\R}{\mathbb{R}}
\newcommand{\Om}{\Omega}
\newcommand{\om}{\omega}

\newcommand{\DO}{\partial\Om}

\newcommand{\vphi}{\varphi}

\newcommand{\CR}{\widetilde{\calT^{n}}}
\newcommand{\calVh}{\mathcal{V}^n}
\newcommand{\calVhint}{\mathcal{V}^n_{\mathrm{int}}}
\newcommand{\calVhext}{\mathcal{V}^n_{\mathrm{ext}}}

\newcommand{\Vn}{V^n}
\newcommand{\Vh}{V_h}
\newcommand{\Vt}{V_{h\tau}}

\newcommand{\VCR}{\widetilde{\Vn}}
\newcommand{\Vnm}{V^{n-1}}
\newcommand{\elCR}{\widetilde{K}}

\newcommand{\RTN}{\bm{RTN}}

\newcommand{\RTNa}{\bm{RTN}_{p_{\ver}}(\Ta)}

\newcommand{\Hdiv}{\bm{H}(\Div)}

\newcommand{\Ta}{\widetilde{\calT^{\ver}}}
\newcommand{\oma}{{\om_{\ver}}}

\newcommand{\Hdivoma}{\bm{H}(\Div,\oma)}
\newcommand{\psia}{\psi_{\ver}}
\newcommand{\hpsia}{\overline{\psi}_{\ver}}

\newcommand{\Va}{\bm{V}_h^\ver}

\newcommand{\Pa}{\calP_{p_{\ver}}(\Ta)}
\newcommand{\Pal}{\calP_{p_{\ver}-1}(\Ta)}

\newcommand{\Vthan}{\bm{V}_{h\tau}^{\ver,n}}

\newcommand{\Pia}{\Pi_{\th}^{\ver,n}}

\newcommand{\stha}{\bm{\sigma}_{\th}^{\ver,n}}

\newcommand{\sth}{\bm{\sigma}_{\th}}

\newcommand{\tautha}{\bm{\tau}_{\th}^{\ver,n}}

\newcommand{\gtautha}{g_{\th}^{\ver,n}}

\newcommand{\Hstar}{H^1_{\dagger}(\oma)}
\newcommand{\Hneg}{[\Hstar]^\prime}

\newcommand{\Uth}{\calI u_{\th}}

\newcommand{\etaEq}{\eta_{\mathrm{F},K}^{n}}
\newcommand{\etaJ}{\eta_{\mathrm{J},K}^n}

\newcommand{\etaOsca}{\eta_{\mathrm{osc}}^{\ver,n}}

\newcommand{\etaOscTh}{\eta_{\mathrm{osc},\th}^n}
\newcommand{\etaOscInit}{\eta_{\mathrm{osc},\mathrm{init}}}

\newcommand{\Ex}{\calE_X}
\newcommand{\Exa}{\calE_X^{\ver,n}}
\newcommand{\Ra}{\calR_{\th}^{\ver,n}}

\newcommand{\YT}{Y_T}

\newcommand{\LH}{L^2(0,T;H^1_0(\Om))}
\newcommand{\HHm}{H^1(0,T;H^{-1}(\Om))}
\newcommand{\Bx}{\calB_X}

\ifIMA
\title{Equilibrated flux a posteriori error estimates in $L^2(H^1)$-norms for high-order discretizations of parabolic problems}

\author{{\sc Alexandre Ern\thanks{Email:alexandre.ern@enpc.fr}, Iain~Smears\thanks{Corresponding author. Email:iain.smears@inria.fr} and Martin~Vohral\'ik\thanks{Email:martin.vorhalik@inria.fr}} \\[2pt] 
Universit\'e Paris-Est, CERMICS (ENPC), 77455 Marne-la-Vall\'{e}e cedex 2, France \\[2pt] \& INRIA Paris, 2 Rue Simone Iff, 75589 Paris, France.}

\shortauthorlist{A.~Ern, I.~Smears and M.~Vohral\'{i}k}

\else

\title{Equilibrated flux a posteriori error estimates in $L^2(H^1)$-norms for high-order discretizations of parabolic problems\thanks{This project has received funding from the European Research Council (ERC) under the European Union's Horizon 2020 research and innovation program (grant agreement No 647134 GATIPOR).}}
\author{Alexandre~Ern\footnotemark[2] \and Iain~Smears\footnotemark[2] \and Martin~Vohral\'ik\footnotemark[2]}

\fi

\begin{document}

\ifIMA
\else
\renewcommand{\thefootnote}{\fnsymbol{footnote}}
\footnotetext[2]{Universit\'e Paris-Est, CERMICS (ENPC), 77455 Marne-la-Vall\'{e}e cedex 2, France \& Inria Paris, 2 rue Simone Iff, 75589 Paris, France, (alexandre.ern@enpc.fr, iain.smears@inria.fr, martin.vohralik@inria.fr).}
\renewcommand{\thefootnote}{\arabic{footnote}}
\fi

\maketitle

\begin{abstract}
{We consider the a posteriori error analysis of fully discrete approximations of parabolic problems based on conforming $hp$-finite element methods in space and an arbitrary order discontinuous Galerkin method in time. Using an equilibrated flux reconstruction, we present a posteriori error estimates yielding guaranteed upper bounds on the $L^2(H^1)$-norm of the error, without unknown constants and without restrictions on the spatial and temporal meshes. It is known from the literature that the analysis of the efficiency of the estimators represents a significant challenge for $L^2(H^1)$-norm estimates. Here we show that the estimator is bounded by the $L^2(H^1)$-norm of the error plus the temporal jumps under the one-sided parabolic condition $h^2 \lesssim \tau$.
This result improves on earlier works that required stronger two-sided hypotheses such as $h \simeq \tau$ or $h^2\simeq \tau$; instead our result now encompasses the practically relevant case for computations and allows for locally refined spatial meshes. The constants in our bounds are robust with respect to the mesh and time-step sizes, the spatial polynomial degrees, and also with respect to refinement and coarsening between time-steps, thereby removing any transition condition.}
\ifIMA{Parabolic partial differential equations, a posteriori error estimates, guaranteed upper bound, polynomial-degree robustness, high-order methods}\fi
\end{abstract}

\ifIMA\else{\noindent\bfseries Key words: }Parabolic partial differential equations, a posteriori error estimates, guaranteed upper bound, polynomial-degree robustness, high-order methods\fi \smallskip

\section{Introduction}\label{sec:intro}
We consider the heat equation
\begin{equation}\label{eq:parabolic}
\begin{aligned}
\p_t u - \Delta u = f & & & \text{in }\Om\times(0,T),\\
 u = 0 & & &\text{on }\DO\times (0,T),\\
 u(0) = u_0 & & &\text{in }\Om,
\end{aligned}
\end{equation}
where $\Om \subset \R^d $, $1\leq d \leq 3$, is a bounded, connected, polytopal open set with Lipschitz boundary, and $T>0$ is the final time. We assume that $f \in L^2(0,T;L^2(\Om))$, and that $u_0\in L^2(\Om)$.
We are interested here in the a posteriori error analysis in the $L^2(H^1)$-norm of fully discrete numerical methods for \eqref{eq:parabolic}. In particular, we consider an arbitrary-order discontinuous Galerkin finite element method (DGFEM) in time, coupled with a conforming $hp$-FEM in space.
We recall that a posteriori error estimates should ideally provide \emph{guaranteed upper bounds} on the error, without unknown constants.
Otherwise, if the estimators constitute an upper bound on the error up to an unknown constant, then we say instead that the estimators are \emph{reliable}.
Furthermore, the estimators should be \emph{locally efficient}, meaning that the local estimators should lie below the error measured in a local space-time neighbourhood, up to a generic constant. Finally, the estimators should ideally be \emph{robust}, with all constants in the bounds being independent of all discretization parameters. Furthermore, on a practical side it is highly desirable that the estimators be locally computable. We refer the reader to~\ccite{Verfurth2013} for an introduction to these concepts. 
Our motivation for considering the heat equation~\eqref{eq:parabolic} as a model problem is that the posteriori error estimates developed in this context serve as a starting point for extensions to diverse applications, for example nonlinear problems~(see \ccite{AmreinWihler2016,DiPietroVohralikSoleiman2015,DolejsiErnVohralik2013,Kreuzer2013}), as well as playing a central role in adaptive algorithms~(see~\ccite{ChenFeng2004,GaspozKreuzerSiebertZiegler2016,KreuzerMollerSchmidtSiebert2012}). For nonconforming discretization methods in space, we refer to~\ccite{ErnVohralik2010,GeorgoulisLakkisVirtanen2011,NicaiseSoualem2005}.

The literature shows that the structure of parabolic problems leads to several outstanding challenges facing the central goals in a posteriori error estimation. In particular, several difficulties arise in the analysis of the efficiency and robustness of the estimators. To explain some of the challenges, first recall that the a posteriori error analysis of parabolic problems admits a range of norms in which to measure the error: for instance, these include the $L^2(H^1)$-norm (see~\ccite{Picasso1998,Verfurth1998}), $L^2(L^2)$-norm (see \ccite{Verfurth1998a}), $L^\infty(L^2)$-norms and $L^\infty(L^\infty)$-norms (see~\ccite{ErikssonJohnson1995}), $L^\infty(L^2)\cap L^2(H^1)$-norms (see~\ccite{LakkisMakridakis2006,MakridakisNochetto2003,SchotzauWihler2010}), and also the $L^2(H^1)\cap H^1(H^{-1})$-norms (see~\ccite{BergamBernardiMghazli2005,ErnVohralik2010,GaspozKreuzerSiebertZiegler2016,NicaiseSoualem2005,Repin2002,Verfurth2003}).
To our knowledge, efficiency results have so far only been attained in the case of the $L^2(H^1)$ norm under restrictions linking mesh and time-step sizes, whereas in the $L^2(H^1)\cap H^1(H^{-1})$ norm, such restrictions have been removed.
It is important to observe that these two functional settings admit an inf-sup theory for the continuous problem that establishes an equivalence between appropriate norms of the error and of the residual. Although no analysis of efficiency is yet available in the setting of other norms, the optimal order of convergence of the estimators has nonetheless been observed in \ccite{LakkisMakridakis2006,MakridakisNochetto2003} for instance.

A posteriori error estimators in the $L^2(H^1)$-norm for a class of nonlinear parabolic problem have been studied in \ccite{Verfurth1998}.
In particular, the analysis in \ccite{Verfurth1998} found that the ratio between the constants in the upper and lower bounds for the error by the estimators depends on $1 + \tau h^{-2} +  \tau^{-1} h^2 + \abs{\log h}$, see \ocite[Prop.~4.1]{Verfurth1998}, where $h$ denotes the spatial mesh size and $\tau$ denotes the time-step size, and thus the efficiency of the estimators is subject to the assumption that $\tau \simeq h^2$. 
\ccite{Picasso1998} studied implicit Euler discretizations of the heat equation: under the assumption that $\tau \simeq h$, he showed that the spatial estimator can be bounded from above by the $L^2(H^1)$-norm of the error plus the temporal jump estimator; in particular, the temporal jump estimator, denoted there by $\varepsilon_K^n$ defined in \ocite[eq.~(2.11)]{Picasso1998}, appears on the right-hand side of the lower bound \ocite[eq.~(2.24)]{Picasso1998}.
In both \ccite{Picasso1998,Verfurth1998}, the two-sided restrictions between the time-step and mesh sizes have the disadvantage of necessarily requiring that the meshes must be quasi-uniform, and thus theoretically prohibiting adaptive refinement.

Starting with \ccite{Verfurth2003}, one approach to removing these two-sided restrictions has been to consider a different functional framework for the a posteriori error analysis, namely by estimating the $L^2(H^1)\cap H^1(H^{-1})$-norm of the error. Part of the justification of this approach is to be found in the observation in \ocite[p.~198, Par.~(5)]{Verfurth2003}, showing that the estimators of \ccite{Picasso1998,Verfurth1998} are upper bounds to not only the $L^2(H^1)$-norm of the error, but also the $L^2(H^1) \cap H^1(H^{-1})$-norm of the error, up to data oscillation. It was then shown in \ccite{Verfurth2003} that these estimators are efficient, locally-in-time yet only globally-in-space, with respect to the $L^2(H^1)\cap H^1(H^{-1})$-norm of the error, without requiring conditions between mesh and time-step sizes; see also \ccite{BergamBernardiMghazli2005}.
Given that the estimators used in both frameworks are the same up to data oscillation, it is of course natural that more general efficiency results are obtainable in when including the $H^1(H^{-1})$ part of the norm, since it allows for the appearance of additional terms on the right-hand side in the efficiency bounds.

Recently, \ifIMA \ccite{ErnSmearsVohralik2016} developed \else we developed in \ccite{ErnSmearsVohralik2016} \fi a posteriori error estimators, based on equilibrated fluxes, for arbitrary order discretizations of parabolic problems within the $L^2(H^1)\cap H^1(H^{-1})$-norm setting, that are guaranteed, locally efficient, and robust. In particular, the analysis does not require any coupling between mesh and time-step sizes, and overcomes the problem of obtaining local-in-space and local-in-time efficiency by considering a natural extension of the $L^2(H^1)\cap H^1(H^{-1})$-norm to the time-nonconforming approximation space. The estimators are robust not only with respect to the mesh and time-step sizes, but also with respect to the polynomial degrees in space and time, and also with respect to mesh coarsening and refinement, thereby removing the so-called transition conditions previously encountered \ccite{Verfurth2003}. These results are built upon the analysis for elliptic problems in~\ccite{BraessPillweinSchoberl2009,ErnVohralik2010,ErnVohralik2015,ErnVohralik2016a}.

In this work, we present a posteriori error estimates for the $L^2(H^1)$-norm of the error, which are based on the same locally computable equilibrated flux as in~\ccite{ErnSmearsVohralik2016}, thereby showing that the same methodology can be used in the $L^2(H^1)$-norm estimates as for the $L^2(H^1)\cap H^1(H^{-1})$-norm.
Our main contributions, presented in Theorem~\ref{thm:X_norm_guaranteed_efficiency} in section~\ref{sec:results} below, include guaranteed upper bounds for the $L^2(H^1)$-norm of the error, and local-in-space-and-time lower bounds for the spatial estimator under the one-sided condition $h^2 \lesssim \tau$. We therefore remove the need for the two-sided conditions encountered previously, and we note that the assumptions in \ccite{Picasso1998,Verfurth1998} were stronger than our assumption. We emphasize that the regime where $h^2 \lesssim \tau$ is the one of practical interest in computations, since implicit methods offer the possibility for large time-steps.
Our lower bound is similar to \ccite{Picasso1998} in at least one respect, namely that the right-hand side of our lower bound includes the temporal jump estimator, since it does not appear possible to show in general that this estimator is locally bounded from above by the $L^2(H^1)$-norm of the error.
Furthermore, we show that the constant of the lower bound is robust with respect to the spatial polynomial degree, and is also robust with respect to refinement and coarsening of the meshes, thereby allowing us to remove the so-called transition conditions. We also show that our results imply local-in-space and local-in-time efficiency when considered in the framework of the augmented norms that were proposed in~\ccite{AkrivisMakridakisNochetto2009,MakridakisNochetto2006,SchotzauWihler2010}.

Our analysis rests upon the following key ingredients. First, in section~\ref{sec:infsup}, we present the inf-sup identity which relates the $L^2(H^1)$-norm of the error to an appropriate dual norm of the residual on test functions in a subspace of $L^2(H^1)\cap H^1(H^{-1})$. After setting the notation for the class of finite element methods in section~\ref{sec:fem}, we recall the construction of the equilibrated flux from our earlier work \ccite{ErnSmearsVohralik2016} in section~\ref{sec:flux_equilibration}. We state the main results in section~\ref{sec:results}. Section~\ref{sec:upper_bound} uses the inf-sup framework to prove the guaranteed upper bounds and the proof of the lower bounds is the subject of section~\ref{sec:efficiency}. It is based on the combination of two key ideas. The first is to take advantage of the semi-discreteness in time of the test functions appearing in the fundamental efficiency result of \ocite[Lemma~8.2]{ErnSmearsVohralik2016} in order to gain control over a negative norm on the time derivatives of the test functions; see Lemma~\ref{lem:X_norm_main_estimate} below. The second idea is to appeal to a specific pointwise-in-time identity for the discontinuous Galerkin time-stepping method, see Lemma~\ref{lem:time_dg_exact} below. Thus, we employ the definition of the numerical scheme for proving the lower bounds, which is somewhat unusual for a posteriori error analysis. The combination of these two ideas then yields the lower bounds stated in section~\ref{sec:results} under the relaxed hypothesis that $h^2 \lesssim \tau$ only.

Throughout this paper, the notation $a\lesssim b$ means that $a\leq C b$, with a generic constant~$C$ that depends possibly on the shape-regularity of the spatial meshes and the space dimension~$\dim$, but is otherwise independent of the mesh-size, time-step size, as well as the spatial and temporal polynomial degrees, or on refinement and coarsening between time-steps.

\section{Inf-sup theory}\label{sec:infsup}
Recall that $\Om \subset \R^d $, $1\leq d \leq 3$ is a bounded, connected, polyhedral open set with Lipschitz boundary. For an arbitrary open subset $\omega\subset \Om$, we use $(\cdot,\cdot)_{\om}$ to denote the $L^2$-inner product for scalar- or vector-valued functions on $\omega$, with associated norm $\norm{\cdot}_{\om}$. In the special case where $\om = \Om$, we drop the subscript notation, i.e.\ $\norm{\cdot}\coloneqq\norm{\cdot}_{\Om}$.

The starting point of the analysis is the weak formulation of problem~\eqref{eq:parabolic} where the time derivative has been cast onto a test function, using integration by parts in time. In particular, the solution space $X$ and test space $\YT$ are defined by
\begin{equation}
\begin{aligned}
	X & \coloneqq \LH,  
	\\  \YT &\coloneqq \{ \vphi \in \LH \cap \HHm,  \vphi(T) =0 \}.
\end{aligned}
\end{equation}
The spaces~$X$ and $\YT$ are equipped with the norms
\begin{equation}\label{eq:XYnorms}
\begin{aligned}
\norm{v}_X^2 & \coloneqq \int_0^T \norm{\nabla v}^2 \,\dd t & & \forall\, v \in X,\\
\norm{\vphi}_{\YT}^2 & \coloneqq \int_0^T \norm{\p_t \vphi}_{H^{-1}(\Om)}^2 + \norm{\nabla \vphi}^2 \,\dd t + \norm{\vphi(0)}^2 & & \forall\, \vphi \in \YT.
\end{aligned}
\end{equation}
Let the bilinear form $\Bx \colon X\times \YT\tends \R$ be defined by
\begin{equation}
\begin{aligned}
\Bx(v,\vphi) \coloneqq \int_0^T - \pair{\p_t \vphi}{v} + (\nabla v,\nabla \vphi) \,\dd t & && \forall\,v\in X,\; \vphi \in \YT,
\end{aligned}
\end{equation}
where $\pair{\cdot}{\cdot}$ denotes the duality pairing between $H^{-1}(\Om)$ and $H^1_0(\Om)$.
Then, problem~\eqref{eq:parabolic} admits the following weak formulation: find $u \in X$ such that
\begin{equation}\label{eq:X_formulation}
\begin{aligned}
\Bx(u,\vphi) = \int_0^T (f,\vphi)\,\dd t + (u_0,\vphi(0)) & && \forall\, \vphi \in \YT.
\end{aligned}
\end{equation}
The well-posedness of~\eqref{eq:X_formulation} is well-known and can be shown by Galerkin's method, see for instance the textbook~\ccite{Wloka1987}.
Note that in this weak formulation, the initial condition $u(0)=u_0$ is expressed as a natural condition, appearing in \eqref{eq:X_formulation}, rather than as an essential condition imposed by the choice of solution space.

\begin{remark}
Problem~\eqref{eq:parabolic} admits an alternative weak formulation where the test space is $X$ and the trial space is $\LH \cap \HHm$
The two formulations possess the same solution, although they lead to different quantitative relations between the norm of the error and of the residual.
\end{remark}

The following result states an inf--sup stability result for the bilinear form~$\Bx$. This inf--sup stability result has the interesting and important property of taking the form of an identity, which is advantageous for the sharpness of a~posteriori error analysis, and shows that the choice of norms for the spaces $X$ and $\YT$ in \eqref{eq:XYnorms} above are optimal.
\begin{theorem}[Inf--sup identity]\label{thm:inf_sup_parabolic}
For every $v \in X$, we have
\begin{equation}
\begin{aligned}
\norm{v}_{X} & = \sup_{\vphi \in Y_T\setminus\{0\}} \frac{ \Bx(v,\vphi) }{\norm{\vphi}_{Y_T}}.  \label{eq:infsup_X}
\end{aligned}
\end{equation}
\end{theorem}
\begin{proof}
The arguments in the proof of~\ocite[Theorem~2.1]{ErnSmearsVohralik2016} can be used to show the following inf-sup identity: for any $\vphi \in \YT$, we have
\begin{equation}\label{eq:adjoint_inf_sup}
\begin{aligned}
\norm{\vphi}_{\YT} = \sup_{v\in X\setminus\{0\}} \frac{\Bx(v,\vphi)}{\norm{v}_X} .
\end{aligned}
\end{equation}
So, \eqref{eq:adjoint_inf_sup} immediately implies the lower bound $\norm{v}_X \geq \sup_{\vphi \in \YT\setminus\{0\}} \Bx(v,\vphi) / \norm{\vphi}_{\YT}$ for any fixed $v\in X$. 
To obtain the converse bound, let $\vphi_* \in \YT$ denote the solution of $\Bx(w,\vphi_*) = \int_0^T (\nabla w, \nabla v) \,\dd t $ for all $w \in X$.
This problem can simply be seen as a backward-in-time parabolic problem with final time condition $\vphi_*(T)=0$. Hence, we have $\norm{v}_X^2 =  \Bx(v,\vphi_*)$ and  \eqref{eq:adjoint_inf_sup} implies that $\norm{\vphi_*}_{\YT} = \norm{v}_X$. This immediately shows that $\norm{v}_X \leq \sup_{\vphi \in \YT\setminus\{0\}} \Bx(v,\vphi) / \norm{\vphi}_{\YT}$, and completes the proof of~\eqref{eq:infsup_X}.
\end{proof}

In order to estimate the error between the solution $u$ of~\eqref{eq:parabolic} and its approximation, we define the residual functional~$\calR_X \colon X\tends [\YT]^\prime$ by
\begin{equation}\label{eq:R_Y_def}
\pair{\calR_X(v)}{\vphi}_{[\YT]^\prime\times \YT} \coloneqq \calB_X(u-v,\vphi)= \int_0^T (f,\vphi) +  \pair{\p_t \vphi}{v} - (\nabla v,\nabla \vphi) \,\dd t + (u_0,\vphi(0)),
\end{equation}
where $v\in X$ and $\vphi \in \YT$, and where the equality follows simply from~\eqref{eq:X_formulation}.
The dual norm of the residual $\norm{\calR_X(v)}_{[\YT]^\prime}$ is naturally defined by
\begin{equation}
\norm{\calR_X(v)}_{[\YT]^\prime}\coloneqq \sup_{\vphi\in \YT\setminus\{0\}} \frac{\pair{\calR_X(v)}{\vphi}}{\norm{\vphi}_{\YT}}.
\end{equation}
Theorem~\ref{thm:inf_sup_parabolic} implies the following \emph{equivalence between the error and dual norm of the residual}: 
\begin{equation}\label{eq:X_error_residual_equivalence}
\begin{aligned}
\norm{u-v}_{X} = \norm{\calR_X(v)}_{[\YT]^\prime} & & &\forall\,v\in X.
\end{aligned}
\end{equation}

\section{Finite element approximation}\label{sec:fem}
The time interval~$(0,T)$ is partitioned into sub-intervals $I_n\coloneqq (t_{n-1},t_n) $, with $1\leq n \leq N$, where it is assumed that $[0,T]=\bigcup_{n=1}^N\overline{I_n}$, and that $\{t_n\}_{n=0}^N$ is strictly increasing with $t_0=0$ and  $t_N = T$.
For each interval $I_n $, we let $\tau_n \coloneqq t_n-t_{n-1}$ denote the local time-step size.
No special assumptions are made about the relative sizes of the time-steps to each other.
A temporal polynomial degree $q_n\geq 0$ is associated to each time-step $I_n$, and we gather all the polynomial degrees in the vector $\bm q = (q_n)_{n=1}^N$.
For a general vector space $V$, we shall write $\calQ_{q_n}\left(I_n;V\right)$ to denote the space of $V$-valued univariate polynomials of degree at most $q_n$ over the time-step interval $I_n$.

\subsection{Meshes}
For each $0\leq n \leq N$, let $\calT^n$ denote a matching simplicial mesh of the domain $\Om$, where we assume shape-regularity of the meshes uniformly with respect to $n$.
We consider here only matching simplicial meshes for simplicity, although we indicate that mixed simplicial--parallelepipedal meshes, possibly containing hanging nodes, can also be treated: see \ccite{DolejsiErnVohralik2016} for instance.
The mesh $\calT^0$ will be used to approximate the initial datum $u_0$.
For each element~$K\in\calT^n$, let $h_K\coloneqq \diam K$ denote the diameter of $K$. We associate a local spatial polynomial degree $p_K\geq 1$ to each $K\in \calT^n$, and we gather all spatial polynomial degrees in the vector $\bm p_n= (p_K)_{K\in\calT^n}$. In order to keep our notation sufficiently simple, the dependence of the local spatial polynomial degrees $p_K$ on the time-step is kept implicit, although we bear in mind that the polynomial degrees may change between time-steps.

\subsection{Approximation spaces}
Given a general matching simplicial mesh $\calT$ and given a vector of polynomial degrees $\bm p=(p_K)_{K\in\calT}$, $p_K\geq 1$ for all $K\in\calT$, we define the $H^1_0(\Om)$-conforming $hp$-finite element space $\Vh(\calT,\bm p)$ by
\begin{equation}\label{eq:conforming_space_def}
\Vh (\calT,\bm p)\coloneqq \left\{v_h \in H^1_0(\Om),\; \eval{v_h}{K} \in \calP_{p_K}(K)\quad\forall\,K\in\calT\right\},
\end{equation}
where $\calP_{p_K}(K)$ denotes the space of polynomials of total degree at most $p_K$ on $K$.
To shorten the notation, let $\Vn \coloneqq V_h(\calT^n,\bm{p}_n)$ for each $0\leq n \leq N$. Let $\Pi_{h} u_0 \in V^0$ denote an approximation to the initial datum $u_0$, a typical choice being the $L^2$-orthogonal projection of $u_0$ onto $V^0$.
Given the collection of time intervals $\{I_n\}_{n=1}^N$, the vector $\bm q$ of temporal polynomial degrees, and the $hp$-finite element spaces $\{\Vn\}_{n=0}^N$, the finite element space $\Vt$ is defined by
\begin{equation}\label{eq:space_time_fem}
\Vt \coloneqq \left\{ v_{\th}|_{(0,T)}\in X,\; \eval{v_{\th}}{I_n} \in \calQ_{q_n}(I_n;\Vn) \quad\forall\, n=1,\dots,N,\; v_{\th}(0)\in V^0 \right\}.
\end{equation}
Functions in $\Vt$ are generally discontinuous with respect to the time-variable at the temporal partition points. We take them to be left-continuous: for all $1\leq n \leq N$, we define $v_{\th}(t_n)$ as the trace at $t_n$ of the restriction $\eval{v_{\th}}{I_n}$. Moreover, functions in $\Vt$ also have a well-defined value at $t_0=0$.
For all $0\leq n < N$, we denote the right-limit of $v_{\th}\in\Vt$ at $t_n$ by $v_{\th}(t_n^+)$.
Then, the temporal jump operators $\lld \cdot \rrd_n$ are defined by
\begin{equation}\label{eq:jump_operators}
\lld v_{\th} \rrd_n \coloneqq v_{\th}(t_n)-v_{\th}(t_n^+), \quad 0\leq n \leq N-1.
\end{equation}

\subsection{Refinement and coarsening}
Similarly to other works, e.g., \ocite[p.~196]{Verfurth2003}, we assume that we have at our disposal a common refinement mesh $\CR$ of $\calT^{n-1}$ and $\calT^n$ for each $1\leq n \leq N$, as well as associated polynomial degrees $\widetilde{\bm{p}}_n=(p_{\elCR})_{\elCR\in\CR}$, such that $\Vnm + \Vn \subset \VCR \coloneqq  \Vh(\CR,\widetilde{\bm{p}}_n) $.
For a function $v_{\th}\in \Vt$, we observe that $\lld v_{\th} \rrd_{n-1} \in \VCR$ for each $1\leq n \leq N$ since $v_{\th}(t_{n-1})\in \Vnm$, $v_{\th}(t_{n-1}^+)\in \Vn$, and $\Vnm + \Vn \subset \VCR$.
It is assumed that the shape-regularity of $\CR$ is equivalent up to constants to those of $\calT^{n-1}$ and $\calT^n$, and that every element $\elCR\in\CR$ is wholly contained in a single element $K^\prime\in\calT^{n-1}$ and a single element $K^{\prime\prime}\in\calT^n$. We emphasize that we do not require any assumptions on the relative coarsening or refinement between successive spaces $\Vnm$ and $\Vn$.
In particular, we do not need the transition condition assumption from~\ocite[p.~196, 201]{Verfurth2003}, which requires a uniform bound on the ratio of element sizes between $\CR$ and $\calT^n$.

\subsection{Numerical method}
The numerical scheme consists of finding  $u_{\th} \in \Vt $ such that $u_{\th}(0)=\Pi_h u_0$, and such that
\begin{equation}\label{eq:num_scheme}
\int_{I_n} (\p_t u_{\th}, v_{\th} ) + (\nabla u_{\th}, \nabla v_{\th}) \,\dd t - \left(\lld u_{\th} \rrd_{n-1}, v_{\th}(t_{n-1}^+) \right) = \int_{I_n} (f,v_{\th}) \,\dd t 
\end{equation}
for all test functions $v_\th \in \calQ_{q_n}(I_n;\Vn)$, for each time-step interval $I_n$, $n=1,\dots, N$.
Here the time derivative $\p_t u_{\th}$ is understood as the piecewise time-derivative on each time-step interval $I_n$.
The numerical solution $u_{\th}\in \Vt$ can thus be obtained by solving the fully discrete problem~\eqref{eq:num_scheme} on each successive time-step. At each time-step, this requires solving a linear system that is symmetric only in the case $q_n=0$; this can be performed efficiently in practice for arbitrary orders following~\ccite{Smears2016}. Note further that the initial condition $u_{\th}(0)=\Pi_h u_0$ does not guarantee that the right-limit $u_{\th}(0^+)$ should equal $\Pi_h u_0$.

\subsection{Reconstruction operator}\label{sec:reconstruction} For each time-step interval $I_n$ and each nonnegative integer $q$, let $L_q^n$ denote the polynomial on $I_n$ obtained by mapping the standard $q$-th Legendre polynomial under an affine transformation of $(-1,1)$ to $I_n$. It follows that $L_q^n(t_n) =1$ for all $q\geq 0$, and $L_q^n(t_{n-1})=(-1)^q$, and that the mapped Legendre polynomials $\{L_q^n\}_{q\geq 0}$ are $L^2$-orthogonal on $I_n$, and satisfy $\int_{I_n}\abs{L_q^n}^2\,\dd t = \frac{\tau_n}{2q+1}$ for all $q\geq 0$.
Following \ccite{MakridakisNochetto2006} (see also \ocite[Remark~2.3]{Smears2016}), we introduce the reconstruction operator $\calI$ defined on $\Vt$ by
\begin{equation}\label{eq:def_radau_reconstruction}
\begin{aligned}
\eval{(\calI v_{\th})}{I_n} \coloneqq \eval{ v_{\th} }{I_n} +  \frac{(-1)^{q_n} }{2} \left( L_{q_n}^n - L_{q_n+1}^n \right) \lld v_{\th} \rrd_{n-1} & & & \forall\,v_{\th}\in\Vt.
\end{aligned}
\end{equation}
It is clear that $\calI$ is a linear operator on $\Vt$.
Furthermore, the definition ensures that $\eval{\calI  v_{\th} }{I_n}(t_n) = v_{\th}(t_n)$, and that $ \eval{\calI v_{\th} }{I_n}(t_{n-1}^+) = v_{\th}(t_{n-1})$ for all $1\leq n\leq N$. 
This implies that $\calI v_{\th}$ is continuous with respect to the temporal variable at the interval partition points $\{t_n\}_{n=0}^{N-1}$ and hence $\calI v_{\th} \in H^1(0,T;H^1_0(\Om))$. Furthermore, $\eval{ \calI v_{\th} }{I_n} \in   \calQ_{q_n+1}\big( I_n;\VCR\big)$ for any $v_{\th}\in\Vt$, where we recall that $\Vnm + \Vn \subset \VCR$.
It is well-known from \ccite{ErnSchieweck2016,MakridakisNochetto2006,Smears2016} that we may rewrite the numerical scheme~\eqref{eq:num_scheme} as
\begin{equation}\label{eq:num_scheme_equiv}
\int_{I_n} (\p_t  \Uth,v_{\th}) + (\nabla u_{\th}, \nabla v_{\th})
\,\dd t = \int_{I_n} (f,v_{\th}) \,\dd t \quad  \forall\,v_\th \in \calQ_{q_n}(I_n;\Vn).
\end{equation}
Note also that $\Uth(0)=\Pi_h u_0$.


\section{Construction of the equilibrated flux}\label{sec:flux_equilibration}

The a posteriori error estimates presented in this paper are based on a discrete and locally computable $\bm{H}(\Div)$-conforming flux~$\sth$ that satisfies the key equilibration property
\begin{equation}\label{eq:sigma_th_equilibration}
\begin{aligned}
\p_t \Uth + \nabla{\cdot} \sth  = f_{\th} & &  &\text{in }\Om\times(0,T),
\end{aligned}
\end{equation}
where $\Uth$ is defined in section~\ref{sec:reconstruction}, and $f_{\th}\approx f$ is an approximation of the data that is defined in~\eqref{eq:f_discrete_approx} below. We call $\sth$ an equilibrated flux.
The construction of $\sth$ given here is exactly the same as in \ccite{ErnSmearsVohralik2016}. This has the practical benefit that a single construction of the equilibrated flux can be used for both a posteriori error estimates in the $L^2(H^1)\cap H^1(H^{-1})$-norm and also in the $L^2(H^1)$-norm.

\subsection{Local mixed finite element spaces}
For each $1\leq n \leq N$, let $\calVh$ denote the set of vertices of the mesh $\calT^n$, where we distinguish the set of interior vertices $\calVhint$ and the set of  boundary vertices $\calVhext$. For each $\ver \in \calVh$, let $\psia$ denote the hat function associated with $\ver$, and let $\oma$ denote the interior of the support of $\psia$, with associated diameter $h_{\oma}$.
Furthermore, let $\Ta$ denote the restriction of the mesh $\CR$ to~$\oma$.
Recalling that the common refinement spaces $\VCR$ were obtained with a vector of polynomial degrees $\widetilde{\bm{p}}_n = (p_{\elCR})_{\elCR\in \CR}$, we associate to each $\ver \in \calVh$ the fixed polynomial degree
\begin{equation}\label{eq:patch_polynomial_degree}
p_{\ver} \coloneqq \max_{\elCR\in \Ta} (p_{\elCR}+1).
\end{equation}
For a polynomial degree $p\geq 0$, let the piecewise polynomial (discontinuous) spaces $\calP_{p}(\Ta)$ and $\RTN_p(\Ta)$ be defined by
\begin{align*}
\calP_{p}(\Ta) &\coloneqq \{ q_h \in L^2(\oma),\quad q_h|_{\elCR} \in \calP_{p}(\elCR)\quad\forall\,\elCR\in\Ta\},
\\  \RTN_p(\Ta) &\coloneqq \{ \bm{v}_h \in \bm{L}^2(\oma),\quad \bm{v}_h|_{\elCR} \in \RTN_{p}(\elCR)\quad\forall\elCR\in\Ta\},
\end{align*}
where  $\RTN_{p}(\elCR) \coloneqq  \bm{\calP}_{p}(\elCR) + \calP_{p}(\elCR)\bm{x}$ denotes the Raviart--Thomas--N\'ed\'elec space of order $p$ on the simplex~$\elCR$.
It is important to notice that whereas the patch $\oma$ is subordinate to the elements of the mesh $\calT^n$ around the vertex $\ver\in\calVh$, the spaces $\calP_{p}(\Ta)$ and $\RTN_p(\Ta)$ are subordinate to the submesh elements in $\Ta$; of course, in the absence of coarsening, this distinction vanishes.
We now introduce the local spatial mixed finite element space $\Va$, defined by
\begin{align*}
\Va & \coloneqq
\begin{cases}
    \left\{\bm{v}_h \in \Hdivoma\cap \RTNa ,\; \bm{v}_h\cdot \bm{n} =0\text{ on }\p\oma \right\} & \text{if }\ver\in\calVhint,\\
    \left\{\bm{v}_h \in \Hdivoma\cap \RTNa  ,\; \bm{v}_h\cdot \bm{n} =0\text{ on }\p\oma\setminus\DO \right\}& \text{if }\ver\in\calVhext.
\end{cases}
\end{align*}
We then define the space-time mixed finite element space 
\begin{equation}\label{eq:spacetime_mixed_space_def}
\begin{aligned}
\Vthan \coloneqq \calQ_{q_n}(I_n;\Va),
\end{aligned}
\end{equation}
where we recall that $\calQ_{q_n}\left(I_n;\Va\right)$ denotes the space of $\Va$-valued univariate polynomials of degree at most $q_n$ over the time-step interval $I_n$.

\subsection{Data approximation}\label{sec:data_approximation}
Our a posteriori error estimates given in section~\ref{sec:results} involve certain approximations of the source term $f$ appearing in \eqref{eq:parabolic}. It is helpful to define these approximations here.
For each $1\leq n \leq N$ and for each $\ver\in\calVh$, let $\Pia$ be the $L^2_{\psia}$-orthogonal projection from $L^2(I_n;L^2_{\psia}(\oma))$ onto $\calQ_{q_n}(I_n;\Pal)$, where $L^2_{\psia}(\oma)$ is the space of measurable functions $v$ on $\oma$ such that $\int_{\oma} \psia \abs{v}^2\,\dd x<\infty$.
In other words, the projection operator $\Pia$ is defined by $\int_{I_n} (\psia \Pia v, q_{\th})_{\oma}\,\dd t = \int_{I_n} (\psia v , q_{\th})_{\oma}\,\dd t$ for all $\, q_{\th} \in \calQ_{q_n}(I_n;\Pal)$.
We adopt the convention that $\Pia v$ is extended by zero from $\oma\times I_n$ to $\Omega\times(0,T)$ for all $v\in L^2(I_n;L^2_{\psia}(\oma))$.
Then, we define $f_{\th}$ by
\begin{equation}\label{eq:f_discrete_approx}
f_{\th} \coloneqq \sum_{n=1}^N\sum_{\ver\in\calVh}\psia \,\Pia f.
\end{equation}
See \ccite{ErnSmearsVohralik2016} for further remarks concerning the approximation properties of $f_{\th}$. In particular, it is shown there that $f_{\th}$ is a data approximation that is at least of same order as the one used in the numerical scheme~\eqref{eq:num_scheme}.

\subsection{Flux reconstruction}\label{sec:flux_reconstruction_def}

For each $1\leq n \leq N$ and each $\ver\in\calVh$, let the scalar function $\gtautha \in \calQ_{q_n}(I_n;\Pa) $ and vector field $\tautha \in \calQ_{q_n}(I_n;\RTNa) $ be defined by
\begin{subequations}\label{eq:tau_g_def}
\begin{align}
\tautha &\coloneqq  \psia \nabla u_{\th}|_{\oma\times I_n},\label{eq:tau_def} \\
\gtautha &\coloneqq \psia\,\left(\Pia  f  - \p_t \Uth\right)|_{\oma\times I_n} - \nabla \psia \cdot \nabla u_{\th}|_{\oma\times I_n}.\label{eq:g_def}
\end{align}
\end{subequations}
For interior vertices, the numerical scheme~\eqref{eq:num_scheme_equiv} implies that
\begin{equation}
\label{eq:gthan_mean_value_zero}
\begin{aligned}
(\gtautha(t),1)_{\oma} = 0 & & & \forall\,t\in I_n.
\end{aligned}
\end{equation}

\begin{definition}[Flux reconstruction]
\label{def:flux_construction_1}Let $u_\th\in \Vt$ be the numerical solution of~\eqref{eq:num_scheme}. For each time-step interval $I_n$ and for each vertex $\ver\in\calV$, let the space $\Vthan$ be defined by \eqref{eq:spacetime_mixed_space_def}.
Let $\gtautha$ and $\tautha$ be defined by \eqref{eq:tau_g_def}.
Let $\stha \in \Vthan$ be defined by
\begin{equation}\label{eq:stha_minimization_def}
\stha \coloneqq \argmin_{\substack{ \bm{v}_h \in \Vthan \\ \nabla{\cdot} \bm{v}_h = \gtautha}}\int_{I_n} \norm{\bm{v}_h + \tautha}_{\oma}^2\,\dd t.
\end{equation}
Then, after extending $\stha$ by zero from $\oma\times I_n$ to $\Om \times (0,T)$ for each $\ver \in \calV$ and for each $1\leq n\leq N$, we define
\begin{equation}\label{eq:flux_reconstruction_1}
\sth\coloneqq \sum_{n=1}^N \sum_{\ver \in \calVh} \stha.
\end{equation}
\end{definition}
Note that $\stha\in \Vthan$ is well-defined for all $\ver\in\calVh$: in particular, for interior vertices $\ver\in\calVhint$, we use~\eqref{eq:gthan_mean_value_zero} to guarantee the compatibility of the datum $\gtautha$ with the constraint $\nabla{\cdot}\stha = \gtautha$.
The following key result is quoted from \ccite{ErnSmearsVohralik2016}.
\begin{theorem}[Equilibration]\label{thm:sigma_th_equilibration}
Let the flux reconstruction~$\sth$ be given by~Definition~{\upshape\ref{def:flux_construction_1}}, and let $f_{\th}$ be defined in \eqref{eq:f_discrete_approx}.
Then $\sth \in L^2(0,T;\Hdiv)$ and the equilibration identity \eqref{eq:sigma_th_equilibration} holds.
\end{theorem}

Moreover, for the purpose of implementation, it is known that on each patch of the mesh and at each time-step, the solution of the minimization problem~\eqref{eq:stha_minimization_def} decouples into $q_n+1$ independent spatial mixed finite element linear systems, which helps to reduce the cost of computing the flux $\sth$.

\section{Main results}\label{sec:results}

We introduce the following a posteriori error estimators and data oscillation terms:
\begin{subequations}\label{eq:estimators}
\begin{align}
[\etaEq]^2 &\coloneqq  \int_{I_n}\norm{\sth  + \nabla u_{\th} }_{K}^2 \,\dd t , \label{eq:etaEq_def} \\
[\etaJ]^2 & \coloneqq \int_{I_n} \norm{\nabla(u_{\th} - \Uth)}_{K}^2 \, \dd t, \label{eq:etaJ_def}\\
 [\etaOscTh]^2 & \coloneqq \frac{1+\sqrt{2}}{2} \int_{I_n} \sum_{\elCR\in\CR} \left[\frac{\tau_n}{\pi}+\frac{h_{\elCR}^2}{\pi^2} \right] \norm{f-f_{\th} }_{\elCR}^2 \,\dd t, \label{eq:etaOscTh_def} \\
\etaOscInit & \coloneqq \norm{u_0-\Pi_h u_0},
\end{align}
\end{subequations}
where, $K\in \calT^n$, $1\leq n \leq N$, the equilibrated flux~$\sth$ is prescribed in Definition~\ref{def:flux_construction_1}, and where the data approximation $f_{\th}$ is defined in section~\ref{sec:data_approximation}.
The total estimator for the error is defined by 
\begin{equation}\label{eq:etaX_def}
[\eta_X]^2 \coloneqq \sum_{n=1}^N \left\{\left[\sum_{K\in\calT^n} \left\{ [\etaEq]^2+[\etaJ]^2\right\}\right]^{\frac{1}{2}} + \etaOscTh \right\}^2 + [\etaOscInit]^2.
\end{equation}
The flux estimator $\etaEq$ and the temporal jump estimator $\etaJ$ are the two main estimators. In particular, the flux estimator $\etaEq$ measures the lack of $\bm{H}(\Div)$-conformity of $\nabla u_{\th}$, and the temporal jump estimator~$\etaJ$ measures the lack of temporal conformity of~$u_{\th}$. 
Indeed, $\etaJ$ is related to the jump $\lld u_{\th} \rrd_{n-1}$, since it was shown in~\ccite{SchotzauWihler2010,ErnSmearsVohralik2016} that $\etaJ$ can be equivalently rewritten as
\begin{equation}\label{eq:jump_equivalent_form}
\etaJ = \sqrt{\tfrac{\tau_n ({q_n}+1)}{(2q_n+1)(2q_n+3)}}\, \norm{\nabla \lld u_{\th} \rrd_{n-1}}_K.
\end{equation}
Given that $\etaEq$ and $\etaJ$ respectively measure the lack of spatial and temporal conformity of the approximate solution, it is common in the literature to call $\etaEq$ the spatial estimator and $\etaJ$ the temporal estimator. However, such terminology must not be interpreted as stating that these estimators bound the errors due respectively to the spatial and temporal discretization.

\begin{theorem}[$X$-norm a posteriori error estimate]\label{thm:X_norm_guaranteed_efficiency}
Let $u \in X$ be the weak solution of \eqref{eq:parabolic}, and let $u_{\th}\in \Vt$ denote the solution of the numerical scheme~\eqref{eq:num_scheme}.
Let~$\eta_X$ be defined by \eqref{eq:etaX_def}. Then, we have the following $X$-norm a posteriori error estimate:
\begin{equation}\label{eq:Xnorm_guaranteed_upper}
\norm{u-u_{\th}}_X \leq \eta_X.
\end{equation}
If $K\in \calT^n$, $1\leq n \leq N$, is an element such that $h_{\oma}^2 \leq \gamma_{\ver} \, \tau_n$ for each $\ver\in\calV_K$, with $\calV_K$ the set of vertices of the element $K$, with some constant $\gamma_{\ver}>0$, where $h_{\oma}$ denotes the diameter of the patch $\oma$, then we have the local lower bound for the flux estimator $\etaEq$
\begin{equation}\label{eq:Xnorm_local_efficiency}
 [\etaEq]^2 \leq C^2_{\gamma_{\ver},q_n} \sum_{\ver\in\calV_{K}} \left\{ \int_{I_n}\norm{\nabla(u-u_{\th})}_{\oma}^2 + \norm{\nabla(u_{\th}-\Uth)}_{\oma}^2 \dd t  + [\etaOsca]^2  \right\},
\end{equation}
where the local data ocillation $\etaOsca$ is defined by
\begin{equation}
\label{eq:tetaOsca_def}
[\etaOsca]^2\coloneqq  \int_{I_n} \norm{ f - \Pia f}_{H^{-1}(\oma)}^2 \dd t.
\end{equation}
Furthermore, under the hypothesis that there exists $\gamma>0$ such that $h_{\oma}^2 \leq \gamma \, \tau_n$ for every $\ver\in\calVh$ and every $1\leq n \leq N$, then we have the global lower bound
\begin{equation}\label{eq:Xnorm_global_efficiency}
\sum_{n=1}^N \sum_{K\in\calT^n} [\etaEq]^2  \leq C_{\gamma,q_n}^2\left\{ \norm{u-u_{\th}}_X^2 + \norm{u_{\th}-\Uth}_X^2 + \sum_{n=1}^N \sum_{\ver\in\calVh} [\etaOsca]^2 \right\}.
\end{equation}
The constants $C_{\gamma_{\ver},q_n}$ in \eqref{eq:Xnorm_local_efficiency} and $C_{\gamma,q_n}$ in~\eqref{eq:Xnorm_global_efficiency} satisfy $C_{\gamma,q_n} \lesssim (q_n+1)^{\frac{1}{2}} + \gamma (q_n+1)^{\frac{5}{2}}$, and may depend on the shape regularity of $\calT^n$ and $\CR$ and on the dimension~$\dim$, but otherwise do not depend on the mesh-size, time-step size, spatial polynomial degrees, or on refinement and coarsening between time-steps.
\end{theorem}
The proof of Theorem~\ref{thm:X_norm_guaranteed_efficiency} is given in several stages throughout the following sections. In the first stage, we give the proof of the upper bound \eqref{eq:Xnorm_guaranteed_upper} immediately after the helpful data oscillation estimate of Lemma~\ref{lem:Xnorm_data_oscillation} below in section~\ref{sec:upper_bound}. In the second stage, we show  the lower bounds \eqref{eq:Xnorm_local_efficiency} and \eqref{eq:Xnorm_global_efficiency} in section~\ref{sec:efficiency}.

\begin{remark}[Bounds for the jump estimator]\label{rem:jump_estimator}
In the local lower bound~\eqref{eq:Xnorm_local_efficiency}, we have $\int_{I_n} \norm{\nabla(u_{\th}-\Uth)}_{\oma}^2\dd t= \sum_{K\subset \oma} [\etaJ]^2$, see also~\eqref{eq:jump_equivalent_form}, where the sum is over all elements $K$ of $\calT^n$ contained in $\oma$. Similarly, in the global lower bound~\eqref{eq:Xnorm_global_efficiency}, the term $\norm{u_{\th}-\Uth}_X^2 = \sum_{n=1}^N\sum_{K\in \calT^n} [\etaJ]^2$ appears.
Thus our result here is comparable to those in \ccite{Picasso1998} where the jump estimator also appears on the right-hand side of the local lower bounds.
The reason for the appearance of this term can be essentially traced back to the lack of Galerkin orthogonality for the temporal reconstruction $\Uth$, see~\eqref{eq:num_scheme_equiv}.
Furthermore, in \ccite{ErnSmearsVohralik2016} it was shown that the (time-local but space-global) jump estimators are bounded from above by the (time-local space-global) $L^2(H^1)\cap H^1(H^{-1})$-norm of the error, up to possible data oscillation.
\end{remark}

\begin{remark}[Comparison with $L^2(H^1)\cap H^1(H^{-1})$-norm estimators]
As pointed out by the remark in \ocite[p.~198, Par.~(5)]{Verfurth2003} concerning the equivalence of residual-based estimators for both $L^2(H^1)$ and $L^2(H^1)\cap H^1(H^{-1})$ norms, it is important to observe that in the absence of data oscillation, the estimator $\eta_X$ defined above in \eqref{eq:etaX_def} is equivalent up to constants to the $L^2(H^1)\cap H^1(H^{-1})$-norm estimator defined in~\ocite[Eq.~(5.10)]{ErnSmearsVohralik2016}. However, an important difference between these estimators concerns the data oscillation. Indeed, it is known since \ccite{Verfurth2003} that $L^2(H^1)\cap H^1(H^{-1})$ estimators generally contain a data oscillation term that can be of same temporal order as the error. By comparison, the data oscillation term \eqref{eq:etaOscTh_def} features an additional half-order with respect to the time-step size. Therefore we expect that the $X$-norm estimator given above may be of special use in situations with significant data oscillation in time.
\end{remark}

Theorem~\ref{thm:X_norm_guaranteed_efficiency} is our main result on a posteriori error estimation of $\norm{u-u_{\th}}_X$. Several authors have also considered various augmented norms and error measures, see e.g.~\ccite{AkrivisMakridakisNochetto2009,MakridakisNochetto2006,SchotzauWihler2010}. For instance, we can define the error measure
\begin{equation}\label{eq:error_measures}
\Ex \coloneqq \max\left\{\norm{u-u_{\th}}_X, \norm{u-\Uth}_X \right\}.
\end{equation}
The choice in $\eqref{eq:error_measures}$ is only one of many possibilities; for instance we could equally well consider $\norm{u-u_{\th}}_X + \norm{u_{\th}-\Uth}_X$.
The interest of this approach is that the  bounds~\eqref{eq:Xnorm_guaranteed_upper}, \eqref{eq:Xnorm_local_efficiency} and~\eqref{eq:Xnorm_global_efficiency} immediately yield a global upper bound and local-in-time and local-in-space efficiency with respect to this error measure, see Corollary~\ref{cor:error_measures} below. 
However, it is important to note that it does not appear possible to show in general an equivalence between $\Ex$ and $\norm{u-u_{\th}}_X$, see Remark~\ref{rem:jump_estimator}. 

\begin{corollary}\label{cor:error_measures}
Let $\Ex$ be defined by \eqref{eq:error_measures}. Then, we have the guaranteed upper bound
\begin{equation}
\Ex \leq 2\, \eta_X,
\end{equation}
If $K\in \calT^n$, $1\leq n \leq N$, is an element such that $h_{\oma}^2 \leq \gamma_{\ver} \, \tau_n$ for each $\ver\in\calV_K$ with some constant $\gamma_{\ver}>0$, where $h_{\oma}$ denotes the diameter of the patch $\oma$, then we have the local efficiency bound
\begin{equation}
[\etaEq]^2 + [\etaJ]^2 \leq C_{\gamma_{\ver},q_n}^2 \sum_{\ver\in\calV_K} \left\{ [\Exa]^2 + [\etaOsca]^2 \right\}.
\end{equation}
where the local error measures $\Exa$, $\ver\in\calVh$, are defined by
\begin{equation}\label{eq:augmented_local_efficiency}
	[\Exa]^2 \coloneqq \max\left\{\int_{I_n}\norm{\nabla(u-u_{\th})}^2_{\oma}\,\dd t, \int_{I_n}\norm{\nabla(u-\Uth)}^2_{\oma}\,\dd t \right\}.
\end{equation}
Furthermore, under the hypothesis that there exists $\gamma>0$ such that $h_{\oma}^2 \leq \gamma \, \tau_n$ for every $\ver\in\calVh$ and every $1\leq n \leq N$, then we have the global efficiency bound
\begin{equation}\label{eq:augmented_global_efficiency}
\sum_{n=1}^N\sum_{K\in\calT^n} [\etaEq]^2 + [\etaJ]^2\leq C_{\gamma,q_n}^2 \left\{[\Ex]^2 + \sum_{n=1}^N \sum_{\ver\in\calVh} [\etaOsca]^2 \right\}.
\end{equation}
\end{corollary}

\section{Proof of the guaranteed upper bound~\eqref{eq:Xnorm_guaranteed_upper}}\label{sec:upper_bound}

We will make use of the following preparatory lemmas.

\begin{lemma}
Let $I_n$ be a given time interval, and let $\vphi \in L^2(I_n;H^1_0(\Om))\cap H^1(I_n;H^{-1}(\Om))$ be an arbitrary function. Let $\vphi^n\in H^1_0(\Om)$, the time-mean value of $\vphi$ over $I_n$, be defined by $\vphi^n \coloneqq \tfrac{1}{\tau_n}\int_{I_n} \vphi\,\dd t$. Then
\begin{subequations}\label{eq:bochner_bounds}
\begin{align}
\int_{I_n} \norm{\nabla \vphi^n}^2\,\dd t & \leq \int_{I_n}\norm{\nabla \vphi}^2\,\dd t, \label{eq:approx_2}
\\ \int_{I_n} \norm{\vphi-\vphi^n}^2 \,\dd t &\leq \frac{\tau_n}{\pi}\left(\int_{I_n} \norm{\p_t \vphi}^2_{H^{-1}(\Om)}\,\dd t\right)^{\frac{1}{2}}\left(\int_{I_n}\norm{\nabla \vphi}^2\,\dd t \right)^{\frac{1}{2}}. \label{eq:approx_1}
\end{align}
\end{subequations}
\end{lemma}
\begin{proof}
The bound \eqref{eq:approx_2} is simply the stability of the $L^2$-projection with respect to time; thus it remains only to show \eqref{eq:approx_1}.
It is well-known that there exists a maximal sequence $\{\psi_k\}_{k=1}^\infty$ that is orthonormal in the $L^2(\Om)$-inner product and orthogonal in the $H^1_0(\Om)$ inner product: i.e.\ $(\psi_k,\psi_j)=\delta_{kj}$ and $(\nabla\psi_k,\nabla\psi_j) = \lambda_{k} \delta_{kj}$, with $\{\lambda_{k}\}_{k=1}^\infty \subset \R_{>0}$. Then, we have $\vphi = \sum_{k=1}^{\infty} \alpha_k \psi_k$ and $\vphi^n = \sum_{k=1}^\infty \overline{\alpha_k} \psi_k$, with real-valued $\alpha_k \in H^1(I_n)$ and $\overline{\alpha_k} = \frac{1}{\tau_n}\int_{I_n} \alpha_k \dd t$. Thus we may use the Poincar\'e inequality for real-valued functions to obtain
\[
\begin{split}
\int_{I_n}\norm{\vphi -\vphi^n}^2\dd t &= \sum_{k=1}^\infty \norm{\alpha_k-\overline{\alpha_k}}_{L^2(I_n)}^2 \leq \frac{\tau_n}{\pi} \sum_{k=1}^\infty \abs{\alpha_k}_{H^1(I_n)}\norm{\alpha_k}_{L^2(I_n)}
 \\ & \leq \frac{\tau_n}{\pi}\left(\sum_{k=1}^\infty \frac{1}{\lambda_k}\abs{\alpha_k}_{H^1(I_n)}^2 \right)^{\frac{1}{2}}\left(\sum_{k=1}^\infty \lambda_k \norm{\alpha_k}_{L^2(I_n)}^2  \right)^{\frac{1}{2}}.
\end{split}
\]
We then deduce \eqref{eq:approx_1} from the identities $\int_{I_n}\norm{\p_t \vphi}_{H^{-1}(\Om)}^2\dd t = \sum_{k=1}^\infty \frac{1}{\lambda_k}\abs{\alpha_k}_{H^1(I_n)}^2$ and $\int_{I_n}\norm{\nabla \vphi}^2\dd t = \sum_{k=1}^\infty \lambda_k \norm{\alpha_k}_{L^2(I_n)}^2$.
\end{proof}

\begin{lemma}\label{lem:Xnorm_data_oscillation}
Let $f\in L^2(0,T;L^2(\Om))$, let $f_{\th}$ be defined by \eqref{eq:f_discrete_approx}, and let $\vphi \in \LH\cap \HHm$ be an arbitrary function. Then, for each $1\leq n \leq N$,
\begin{equation}\label{eq:X_data_oscillation_estimate}
\begin{aligned}
\Labs{\int_{I_n} (f-f_{\th},\vphi)\,\dd t }  \leq \etaOscTh \left(
\int_{I_n} \norm{\p_t \vphi}_{H^{-1}(\Om)}^2+\norm{\nabla \vphi}^2 \,\dd t \right)^{\frac{1}{2}} .
\end{aligned}
\end{equation}
\end{lemma}
\begin{proof}
For a given function $\vphi\in \LH\cap \HHm$, we define $\vphi^n$ the time-mean value of $\vphi$ over $I_n$ as $\vphi^n \coloneqq \tfrac{1}{\tau_n}\int_{I_n} \vphi\,\dd t\in H^1_0(\Om)$, and we define $\vphi^n_{\elCR}$ the space-mean value of $\vphi^n$ over $\elCR$ as $\vphi^n_{\elCR}|_{\elCR} \coloneqq \tfrac{1}{\abs{\elCR}}\int_{\elCR}\vphi^n \dd x$, where $1\leq n\leq N$ and $\elCR\in\CR$.
Now, we note that the definition of $f_{\th}$ in \eqref{eq:f_discrete_approx} implies that $f-f_{\th}$ has zero mean value over each space-time element $\elCR\times I_n$. Therefore, we obtain
\[
\int_{I_n}(f-f_{\th},\vphi)\,\dd t = \int_{I_n}(f-f_{\th},\vphi-\vphi^n) + \sum_{\elCR\in\CR} (f-f_{\th},\vphi^n-\vphi^n_{\elCR})_{\elCR}\,\dd t\eqqcolon A+B.
\]
Then, we apply the bounds \eqref{eq:approx_2} and \eqref{eq:approx_1} to obtain
\begin{align*}
\abs{A} &\leq \left(\int_{I_n} \sum_{\elCR\in\CR} \frac{\tau_{n}}{\pi} \norm{f-f_{\th}}_{\elCR}^2\,\dd t\right)^{\frac{1}{2}}  \left(\int_{I_n}\norm{\p_t \vphi}^2_{H^{-1}(\Om)}\,\dd t\right)^{\frac{1}{4}}\left(\int_{I_n}\norm{\nabla \vphi}^2\,\dd t \right)^{\frac{1}{4}},
\\
\abs{B} &\leq \left(\int_{I_n} \sum_{\elCR\in\CR} \frac{h^2_{\elCR}}{\pi^2} \norm{f-f_{\th}}_{\elCR}^2\,\dd t\right)^{\frac{1}{2}} \left(\int_{I_n}\norm{\nabla \vphi}^2\,\dd t\right)^{\frac{1}{2}}.
\end{align*}
Then, the Cauchy--Schwarz inequality and the Young inequality $ ab + b^2 \leq \frac{1+\sqrt{2}}{2}(a^2 + b^2)$ for all $a$, $b\in \R$, imply that the bound \eqref{eq:X_data_oscillation_estimate} holds.
\end{proof}

\paragraph{Proof of the upper bound \eqref{eq:Xnorm_guaranteed_upper}.}
Recall from~\eqref{eq:X_error_residual_equivalence} on the equivalence of norms and residuals that $\norm{u-u_{\th}}_X=\norm{\calR_X(u_{\th})}_{[Y_T]'}$, so we turn our attention to bounding $\pair{\calR_X(u_{\th})}{\vphi}$ for an arbitrary test function $\vphi \in Y_T$.
By adding and subtracting $\int_{0}^T (\p_t \Uth + \nabla\cdot \sth,\vphi)\,\dd t$ and recalling the flux equilibration identity~\eqref{eq:sigma_th_equilibration}, we get\begin{multline}\label{eq:Xnorm_upper_1}
\pair{\calR_X(u_{\th})}{\vphi} = \int_{0}^T(f-f_{\th},\vphi)+\pair{\p_t \vphi}{u_{\th}-\Uth}-(\sth+\nabla u_{\th},\nabla \vphi)\,\dd t \\+ (u_0-\Pi_h u_0,\vphi(0)),
\end{multline}
where we have used integration by parts with respect to time for the time derivative $\p_t \Uth$, noting that $\Uth(0)=\Pi_h u_0$ and that $\vphi(T)=0$, and also where we have used integration by parts over $\Om$ for the flux $\sth \in L^2(0,T;H(\Div,\Om))$. Employing the shorthand notation $\norm{\vphi}_{Y(I_n)}^2\coloneqq \int_{I_n}\norm{\p_t\vphi}_{H^{-1}(\Om)}^2 +\norm{\nabla\vphi}^2\,\dd t$, we then use Lemma~\ref{lem:Xnorm_data_oscillation} and the Cauchy--Schwarz inequality to bound 
\begin{equation}\label{eq:Xnorm_upper_2}
\begin{split}
& \int_{0}^T(f-f_{\th},\vphi)+\pair{\p_t \vphi}{u_{\th}-\Uth}-(\sth+\nabla u_{\th},\nabla \vphi)\,\dd t
\\ & \leq \sum_{n=1}^N \left\{\left[\int_{I_n}  \norm{\sth + \nabla u_{\th}}^2+ \norm{\nabla(u_{\th}-\Uth)}^2\,\dd t  \right]^{\frac{1}{2}} +  \etaOscTh  \right\}\norm{\vphi}_{Y(I_n)} 
\\ & = \sum_{n=1}^N \left\{\left[\sum_{K\in\calT^n} \left\{ [\etaEq]^2+[\etaJ]^2\right\}\right]^{\frac{1}{2}} + \etaOscTh \right\} \norm{\vphi}_{Y(I_n)}.
\end{split}
\end{equation}
We then combine \eqref{eq:Xnorm_upper_1} and \eqref{eq:Xnorm_upper_2} with the Cauchy--Schwarz inequality to find that $\pair{\calR_X(u_{\th})}{\vphi} \leq \eta_X \norm{\vphi}_{Y_T}$; since $\vphi\in Y_T$ was arbitrary, we obtain $\norm{u-u_{\th}}_X\leq \eta_X$ as a result of \eqref{eq:X_error_residual_equivalence}, thereby completing the proof of \eqref{eq:Xnorm_guaranteed_upper}. \qed

\section{Proof of the bounds~\eqref{eq:Xnorm_local_efficiency} and \eqref{eq:Xnorm_global_efficiency}}\label{sec:efficiency}

We start by observing that $\sth|_{K\times I_n} = \sum_{\ver\in\calV_K} \stha|_{K\times I_n}$, and thus
\begin{equation}
\int_{I_n} [\etaEq]^2 \,\dd t = \int_{I_n} \norm { {\textstyle\sum}_{\ver\in \calV_K} (\stha + \psia \nabla u_{\th})}_{K}^2\,\dd t  \leq \abs{\calV_K}   \sum_{\ver\in\calV_K} \int_{I_n}\norm{\stha + \psia \nabla u_{\th}}_K^2 \,\dd t, \label{eq:X_norm_local_efficiency_1}
 \end{equation}
where we recall that $\calV_K$ stands for the vertices of the element $K$ and $\abs{\calV_K}$ stands for its cardinality.
We shall now bound the right-hand side of \eqref{eq:X_norm_local_efficiency_1}.
For each $1\leq n\leq N$ and each $\ver \in \calVh$, we introduce the patch residual functional $\Ra \colon L^2(I_n,H^1_0(\oma))\tends\R$ defined by
\begin{equation}\label{eq:patch_residual}
\begin{aligned}
\pair{\Ra}{\vphi}  =\int_{I_n} \big(\Pia f - \p_t \Uth ,\vphi\big)_{\oma}  -\big(\nabla u_{\th},\nabla \vphi \big)_\oma\dd t & & & \forall\,\vphi \in L^2(I_n;H^1_0(\oma)).
\end{aligned}
\end{equation}

The essential result that forms the starting point for our analysis is the following abstract efficiency result first shown in \ocite[Lemma~8.2]{ErnSmearsVohralik2016}, which is an application of a more general underlying key result in~\ocite[Theorem~1.2]{ErnSmearsVohralik2016c} concerning the existence of polynomial-degree robust liftings of piecewise polynomial data into discrete subspaces of $\Hdiv$, which itself is based on the fundamental results of \ccite{CostabelMcintosh2010,BraessPillweinSchoberl2009}. 
\begin{lemma}[Space-time stability bound]\label{lem:flux_reconstruction_stability}
Let $\stha$ denote the patch-wise flux reconstructions of Definition~{\upshape\ref{def:flux_construction_1}}, and let $\Ra$ denote the local patch residual defined by \eqref{eq:patch_residual}. Then, we have
\begin{equation}\label{eq:flux_reconstruction_stability}
\left(\int_{I_n}\norm{\stha + \psia \nabla u_{\th}}_\oma^2\,\dd t\right)^{\frac{1}{2}}  \lesssim \sup_{\vphi \in \calQ_{q_n}(I_n;H^1_0(\oma))\setminus\{0\}} \frac{\pair{\Ra}{\vphi}}{\left(\int_{I_n}\norm{\nabla \vphi}_{\oma}^2\,\dd t\right)^{\frac{1}{2}}},
\end{equation}
where $\calQ_{q_n}(I_n;H^1_0(\oma))$ denotes the space of $H^1_0(\oma)$-valued univariate polynomials of degree at most $q_n$ on $I_n$. In particular, the constant in \eqref{eq:flux_reconstruction_stability} does not depend on the mesh-size, time-step size, spatial and temporal polynomial degrees, or on refinement and coarsening between time-steps.
\end{lemma}

As explained above in the introduction, our analysis of the efficiency of the equilibrated flux estimator $\etaEq$ relies on two original ideas. We now detail the first one, which is based on the key observation that the set of test functions appearing in \eqref{eq:flux_reconstruction_stability} are polynomials with respect to the time variable. Hence, in order to obtain estimates on the efficiency of the estimators with respect to the $X$-norm of the error, we shall show that the set of test functions appearing in~\eqref{eq:flux_reconstruction_stability} can be restricted to functions vanishing at the end-points of the time interval and thereby lying in the test space $\YT$ through a bubble-in-time argument, provided that $h_{\oma}^2 \lesssim \tau_n$.

We start by defining the space $\Hstar$ through
\begin{equation}\label{eq:Hstar}
\Hstar 
\coloneqq \begin{cases}
 \{ v \in H^1(\oma), \quad (v,\psia)_{\oma}=0 \} & \text{if }\ver\in\calVhint, \\
 \{ v \in H^1(\oma), \quad \eval{v}{\p\oma \cap \DO} =0 \} & \text{if }\ver\in\calVhext.
 \end{cases}
\end{equation}
Recall that the dual norm $\norm{\cdot}_{\Hneg}$ of $\Hstar$ is defined by $\norm{\Phi}_{\Hneg}=\sup\pair{\Phi}{v}$, where the supremum is taken among all test functions $v\in \Hstar$ such that $\norm{\nabla v}_{\oma}=1$.
The motivation for working with the space $\Hstar$ is that the $\psia$-weighted mean value of the function $u-\Uth$ possesses special properties derived from the numerical scheme; in particular, see Lemma~\ref{lem:time_dg_exact} and the discussion surrounding~\eqref{eq:auxiliary_error_function} below.

\begin{lemma}[Stability with test functions vanishing at both endpoints of $I_n$]\label{lem:X_norm_main_estimate}
Let $\ver\in\calT^n$, $1\leq n \leq N$, and suppose that there exists a constant $\gamma$ such that the patch diameter $h_{\oma}$ and $\tau_n$ satisfy $ h_{\oma}^2/\tau_n \leq \gamma_{\ver} $. Then,
\ifIMA
\begin{equation}\label{eq:X_norm_efficiency_main_estimate}
\left(\int_{I_n}\norm{\stha + \psia \nabla u_{\th}}_\oma^2\,\dd  t\right)^{\frac{1}{2}} \leq C_{\gamma_{\ver},q_n} \sup_{\substack{\vphi \in \calQ_{q_n+2}(I_n;H^1_0(\oma)) \\ \cap H^1_0(I_n;H^1_0(\oma))}} \frac{\pair{\Ra}{\vphi}}{\left( \int_{I_n}\norm{\p_t \vphi}_{\Hneg}^2+\norm{\nabla \vphi}_{\oma}^2\,\dd  t\right)^{\frac{1}{2}}},
\end{equation}
\else
\begin{multline}\label{eq:X_norm_efficiency_main_estimate}
\left(\int_{I_n}\norm{\stha + \psia \nabla u_{\th}}_\oma^2\,\dd  t\right)^{\frac{1}{2}} \\  \leq C_{\gamma_{\ver},q_n} \sup_{\substack{\vphi \in \calQ_{q_n+2}(I_n;H^1_0(\oma)) \\ \cap H^1_0(I_n;H^1_0(\oma))}} \frac{\pair{\Ra}{\vphi}}{\left( \int_{I_n}\norm{\p_t \vphi}_{\Hneg}^2+\norm{\nabla \vphi}_{\oma}^2\,\dd  t\right)^{\frac{1}{2}}},
\end{multline}
\fi
where $H^1_0(I_n;H^1_0(\oma))$ denotes the space of functions in $H^1(I_n;H^1_0(\oma))$ that vanish at both endpoints $t_{n-1}$ and $t_n$ of the time interval $I_n$. In particular, the constant $C_{\gamma_{\ver},q_n}$ in~\eqref{eq:X_norm_efficiency_main_estimate} satisfies $C_{\gamma_{\ver},q_n} \lesssim (q_n+1)^{\frac{1}{2}} + \gamma_{\ver} (q_n+1)^{\frac{5}{2}}$, and may depend on the shape regularity of $\calT^n$ and $\CR$ and on the space dimension~$\dim$, but otherwise does not depend on the mesh-size, time-step size, spatial polynomial degrees, or on refinement and coarsening between time-steps.
\end{lemma}

\begin{proof}
The starting point for the proof is Lemma~\ref{lem:flux_reconstruction_stability}. Keeping in mind the right-hand side of \eqref{eq:flux_reconstruction_stability}, for each $\vphi\in \calQ_{q_n}(I_n;H^1_0(\oma))$, we shall construct a new function $\vphi_*\in \calQ_{q_n+2}(I_n;H^1_0(\oma))$ defined by
\[
\vphi_* \coloneqq \vphi - \vphi(t_{n-1}^+) \frac{(-1)^{q_n+1}}{2}(L^n_{q_n+1}-L^n_{q_n+2}) - \vphi(t_n) \frac{1}{2}(L^n_{q_n+1}+L^n_{q_n+2}).
\]
It follows from the fact that $L^n_q(t_{n-1})=(-1)^q$ and that $L^n_q(t_n)=1$ for all $q\geq 0$ that $\vphi_*(t_{n-1}^+)=\vphi_*(t_n)=0$ and hence $\vphi_* \in H^1_0(I_n;H^1_0(\oma))$. Recalling that the functions $\Pia f$, $ \p_t \Uth$, and $\nabla u_{\th}$ appearing in \eqref{eq:patch_residual} are polynomials of degree at most $q_n$ in time, it also follows from the orthogonality of the Legendre polynomials that
\[
\pair{\Ra}{\vphi_*} = \pair{\Ra}{\vphi}.
\]
It is then seen that we shall obtain \eqref{eq:X_norm_efficiency_main_estimate} as a result of \eqref{eq:flux_reconstruction_stability} provided that we can bound $\int_{I_n}\norm{\p_t \vphi_*}_{\Hneg}^2+\norm{\nabla \vphi_*}_{\oma}^2\,\dd t$ in terms of $\int_{I_n} \norm{\nabla \vphi}_{\oma}^2 \,\dd t$.  First, the triangle inequality and the properties of the Legendre polynomials imply that
\begin{equation}\label{eq:X_norm_efficiency_1}
\int_{I_n}\norm{\nabla\vphi_*}_{\oma}^2\,\dd  t \lesssim \int_{I_n} \norm{\nabla\vphi}_{\oma}^2 \,\dd  t + \tfrac{\tau_n}{q_n+1}\left(\norm{\nabla\vphi(t_{n-1})}_{\oma}^2+\norm{\nabla\vphi(t_n)}_{\oma}^2\right),
\end{equation}
where the constant is independent of all other quantities.
Now, the key point is that we have the inverse inequality
\begin{equation}\label{eq:discrete_inverse_inequalities}
 \max_{t\in I_n}\norm{\nabla\vphi(t)}_{\oma}^2 \lesssim \tfrac{(q_n+1)^2}{\tau_n} \int_{I_n}\norm{\nabla\vphi}_{\oma}^2\,\dd t,
\end{equation}
where the constant is independent of all other quantities since $\vphi\in \calQ_{q_n}(I_n;H^1_0(\oma))$ is discrete with respect to time.
Note in particular that the inverse inequality is valid even though $\calQ_{q_n}(I_n;H^1_0(\oma))$ is itself an infinite dimensional space, see Remark~\ref{rem:inverse_inequality} below.
Therefore, we find from \eqref{eq:X_norm_efficiency_1} and \eqref{eq:discrete_inverse_inequalities} that 
\begin{equation}
\int_{I_n}\norm{\nabla\vphi_*}_{\oma}^2\,\dd t \lesssim (q_n+1) \int_{I_n}\norm{\nabla\vphi}_{\oma}^2\,\dd t.
\end{equation}

To bound $\int_{I_n} \norm{\p_t \vphi_*}^2_{\Hneg}\dd t$, we recall that $\vphi_*(t) \in H^1_0(\oma)$ for all $t\in I_n$, and therefore satisfies the Poincar\'e inequality $\norm{\vphi_*(t)}_{\oma} \lesssim h_{\oma} \norm{\nabla \vphi_*(t)}_{\oma}$ for all $t\in I_n$.
Furthermore, we also have a similar Poincar\'e inequality for all test functions $v\in \Hstar$. Therefore, we find that $\norm{\vphi_*(t)}_{\Hneg} \lesssim h_{\oma}^2 \norm{\nabla \vphi_*(t)}_{\oma}$, for all $t\in I_n$. Thus, we obtain, using an inverse inequality in time,
\begin{multline*}
\int_{I_n}\norm{\p_t\vphi_*}_{\Hneg}^2\,\dd t \lesssim \tfrac{(q_n+1)^4}{\tau_n^2}\int_{I_n}\norm{\vphi_*}_{\Hneg}^2 \,\dd t  
 \\ \lesssim  \tfrac{(q_n+1)^4 h_{\oma}^4 }{\tau_n^2} \int_{I_n}\norm{\nabla \vphi_{*}}_{\oma}^2\,\dd t
  \lesssim \gamma_{\ver}^2 (q_n+1)^5 \int_{I_n}\norm{\nabla \vphi}_{\oma}^2\,\dd  t,
\end{multline*}
where we have used the hypothesis that $ h_{\oma}^2/ \tau_n\leq \gamma_{\ver}$ in the last inequality. Hence, we have shown that
\begin{equation}\label{eq:X_norm_efficiency_2}
\int_{I_n} \norm{\p_t \vphi_*}_{\Hneg}^2 + \norm{\nabla \vphi_*}_{\oma}^2 \,\dd  t \leq C^2_{\gamma_{\ver},q_n} \int_{I_n}\norm{\nabla \vphi}_{\oma}^2\,\dd  t,
\end{equation}
where the constant $C_{\gamma_{\ver},q_n} \lesssim (q_n+1)^{\frac{1}{2}} + \gamma_{\ver} (q_n+1)^{\frac{5}{2}}$.
The bound~\eqref{eq:X_norm_efficiency_main_estimate} then follows from \eqref{eq:X_norm_efficiency_2} and the identity $\pair{\Ra}{\vphi}=\pair{\Ra}{\vphi_*}$ given above.
\end{proof}

\begin{remark}[Inverse inequality]\label{rem:inverse_inequality}
The proof of the inverse inequalities appearing above in \eqref{eq:discrete_inverse_inequalities} can be found simply by expanding the function $\vphi$ in any orthogonal basis $\{ \psi_k \}_{k=1}^{\infty}$ of $H^1_0(\oma)$ as $\vphi(t) = \sum_{k=1}^\infty c_k(t) \psi_k$, where the coefficient functions $c_k$ are real-valued polynomials of degree at most $q_n$, for all $k\geq 1$, and then by applying coefficient-wise known inverse inequalities for real-valued functions.
\end{remark}

Lemma~\ref{lem:X_norm_main_estimate} constitutes the first step towards the local lower bound \eqref{eq:Xnorm_local_efficiency}. In particular, we see that the test functions in \eqref{eq:X_norm_efficiency_main_estimate} are bounded in $H^1(I_n;\Hneg)$ norm. In order to exploit this property, we use a second key idea for our analysis, which is to employ the following special property of the time-discretization scheme. Together, these two ingredients allows us to obtain the lower bounds assuming only that $h^2 \lesssim \tau$, rather than the stronger requirements used in \ccite{Picasso1998,Verfurth1998}.

\begin{lemma}[Pointwise identity]\label{lem:time_dg_exact}
For each $1\leq n\leq N$ and each interior vertex $\ver\in\calVhint$, the functions $\Uth$ and $u_{\th}$ satisfy
\begin{equation}\label{eq:scheme_pointwise_identity}
\begin{aligned}
\pair{\p_t\Uth}{\psia} + (\nabla u_{\th},\nabla \psia)= (\Pia f,\psia)  & & &\text{pointwise in } I_n,
\end{aligned}
\end{equation}
where $\Pia f$ was defined in section~\ref{sec:data_approximation}.
\end{lemma}
\begin{proof}
Since $\ver\in\calVhint$, it follows that  $\phi \psia\in \calQ_{q_n}(I_n;\Vn)$ for any polynomial $\phi$ in time of degree at most $q_n$ over $I_n$.
Therefore, the numerical scheme~\eqref{eq:num_scheme_equiv} implies that, for any real-valued polynomial $\phi$ in time of degree at most $q_n$,
\[
\int_{I_n}\phi \left[(f,\psia)-(\p_t \Uth,\psia)-(\nabla u_{\th},\nabla \psia)\right]\dd t =0.
\]
Furthermore, the definition of $\Pia$ implies that $\int_{I_n} \phi (f,\psia) \dd t = \int_{I_n} \phi (\Pia f, \psia) \dd t$ for any real-valued polynomial  $\phi$ in time of degree at most $q_n$. Since the function $t\mapsto (\p_t \Uth (t),\psia)+(\nabla u_{\th}(t),\nabla \psia)- (\Pia f(t),\psia)$ is a real-valued polynomial of degree at most $q_n$ over $I_n$, it follows that it vanishes everywhere in $I_n$. We therefore obtain \eqref{eq:scheme_pointwise_identity}.
\qquad\end{proof}

We now give the proof of the~bounds~\eqref{eq:Xnorm_local_efficiency} and \eqref{eq:Xnorm_global_efficiency} under the hypothesis stated in Theorem~\ref{thm:X_norm_guaranteed_efficiency}.

\paragraph{Proof of the bounds~\eqref{eq:Xnorm_local_efficiency} and \eqref{eq:Xnorm_global_efficiency}}
The proof consists in bounding the right-hand side of \eqref{eq:X_norm_efficiency_main_estimate} so as to show that, for each $\ver\in\calVh$, we have the bound
\begin{multline}\label{eq:X_norm_local_efficiency_2}
\int_{I_n} \norm{\stha + \psia \nabla u_{\th}}_{\oma}^2\,\dd t 
\leq C_{\gamma_{\ver},q_n}^2 \left\{ \int_{I_n} \norm{\nabla(u-\Uth)}_{\oma}^2 + \norm{\nabla(u-u_{\th})}_{\oma}^2\,\dd t \right. \\ 
\left. +  \int_{I_n}\norm{ f - \Pia f}_{H^{-1}(\oma)}^2 \,\dd t \right\},
\end{multline}
where $C_{\gamma_{\ver},q_n} \lesssim (q_n+1)^{\frac{1}{2}} + \gamma (q_n+1)^{\frac{5}{2}}$.
Then, once \eqref{eq:X_norm_local_efficiency_2} is known, it is then straightforward to show~\eqref{eq:Xnorm_local_efficiency} and~\eqref{eq:Xnorm_global_efficiency} from~\eqref{eq:X_norm_local_efficiency_1}.

To show~\eqref{eq:X_norm_local_efficiency_2}, we will treat first the more difficult case where $\ver\in \calVhint$ is an interior node.
It will be convenient to denote $\hpsia \coloneqq \psia / \norm{\psia}_{L^1(\oma)}$ the renormalized hat function associated with $\ver$. Let $\vphi\in \calQ_{q_n+2}(I_n;H^1_0(\oma))\cap H^1_0(I_n;H^1_0(\oma))$ be a fixed but arbitrary test function, such that $\int_{I_n}\norm{\p_t \vphi}_{\Hneg}^2+\norm{\nabla \vphi}_{\oma}^2\,\dd t=1$.
It follows that the zero-extension of $\vphi$ to $\Om\times (0,T)$ belongs to $\YT$, and therefore, we may use the weak formulation \eqref{eq:X_formulation} in the definition of $\Ra$ from \eqref{eq:patch_residual} to find that
\begin{equation}\label{eq:patch_residual_1}
\pair{\Ra}{\vphi} = \int_{I_n}-( u-\Uth, \p_t \vphi)_{\oma} + (\nabla(u-u_{\th}),\nabla \vphi)_{\oma} + (\Pia f - f,\vphi)_{\oma}\,\dd t.
\end{equation}
Note that, in general, $u-\Uth$ fails to belong to $\Hstar$ when $\ver\in\calVhint$ is an interior node because we can not generally guarantee that $(u-\Uth,\psia)_{\oma}=0$ a.e.~in time; thus, $\abs{(u-\Uth,\p_t \vphi)_{\oma}}\not\leq \norm{\nabla (u-\Uth)}_{\oma}\norm{\p_t \vphi}_{\Hneg}$ in general. To overcome this obstacle, we introduce the auxiliary function 
\begin{equation}\label{eq:auxiliary_error_function}
e_{\ver} \coloneqq u-\Uth - (u-\Uth,\hpsia)_{\oma},
\end{equation}
that is, we subtract the $\hpsia$-weighted average of $u-\Uth$ from $u-\Uth$. It follows from the definition that $e_{\ver}(t)\in \Hstar$ and that $\norm{\nabla e_{\ver}(t)}_{\oma} = \norm{\nabla(u-\Uth)(t)}_{\oma}$ for almost all $t\in I_n$. We now show how to reformulate the patch residual $\pair{\Ra}{\vphi}$ in terms of the auxiliary function $e_{\ver}$.
First, we may choose the test function $ \hpsia (\vphi,1)_{\oma} \in \YT$ in \eqref{eq:X_formulation}, and use Fubini's theorem and linearity of integration to find that
\begin{equation}\label{eq:mean_value_term_1}
\begin{split}
\int_{I_n} - ( (u,\hpsia)_{\oma} , \p_t \vphi)_{\oma} \dd t &= \int_{I_n} - \pair{u}{\p_t(\hpsia (\vphi,1)_{\oma})} \dd t
\\ & = \int_{I_n} (f, \hpsia)_{\oma}(\vphi,1)_{\oma}- (\nabla u,\nabla \hpsia)_{\oma}(\vphi,1)_{\oma} \dd t.
\end{split}
\end{equation}
Next, we multiply \eqref{eq:scheme_pointwise_identity} by $(\vphi,1)_{\oma}$ and integrate by parts over $I_n$ and obtain
\begin{equation}\label{eq:mean_value_term_2}
\int_{I_n} - ( (\Uth,\hpsia)_{\oma}, \p_t \vphi )_{\oma} \dd t = \int_{I_n} (\Pia f,\hpsia)_{\oma} (\vphi,1)_{\oma} - (\nabla u_{\th},\nabla\hpsia)_{\oma} (\vphi,1)_{\oma} \dd t.
\end{equation}
The combination of \eqref{eq:patch_residual_1} with \eqref{eq:mean_value_term_1} and \eqref{eq:mean_value_term_2} shows that $\pair{\Ra}{\vphi}=\sum_{i=1}^5 R_i$, where the quantities $R_i$, $1\leq i \leq 5$, are defined by
\begin{gather*}
 R_1 \coloneqq \int_{I_n} -  ( e_{\ver}  , \p_t \vphi  )_{\oma}\,\dd t ,
\\  R_2 \coloneqq \int_{I_n} (\nabla(u-u_{\th}),\nabla\vphi )_{\oma}\,\dd t,
 \qquad R_3 \coloneqq - \int_{I_n} (\nabla(u-u_{\th}),\nabla \hpsia)_{\oma}(\vphi,1)_{\oma}\,\dd t,
\\  R_4 \coloneqq \int_{I_n} (f- \Pia f,\hpsia)_{\oma}(\vphi,1)_{\oma}\,\dd t,
  \qquad R_5 \coloneqq - \int_{I_n} ( f - \Pia f,\vphi)_{\oma}\,\dd t .
\end{gather*}
Using the fact that $\int_{I_n} \norm{\p_t \vphi}^2_{\Hneg}\,\dd t\leq 1$, where we recall that $\Hstar$ is defined in~\eqref{eq:Hstar}, and that $\norm{\nabla e_{\ver}}_{\oma}=\norm{\nabla(u-\Uth)}_{\oma}$, we find that $\abs{R_1}^2\leq  \int_{I_n}\norm{\nabla(u-\Uth)}_{\oma}^2\,\dd t$.
Next, we find that $\abs{R_2}^2 \leq \int_{I_n}\norm{\nabla(u-u_{\th})}_{\oma}^2\,\dd t$. To bound $R_3$ and $R_4$, we apply the Cauchy--Schwarz inequality and use the Poincar\'e inequality on $H^1_0(\oma)$ to obtain
\[
\abs{R_3}^2 + \abs{R_4}^2 \lesssim \int_{I_n} \frac{h_{\oma}^2 \abs{\oma} \norm{\nabla \psia}_{\oma}^2}{\norm{\psia}_{L^1(\oma)}^2} \left[ \norm{\nabla(u-u_{\th})}_{\oma}^2+\norm{f-\Pia f}_{H^{-1}(\oma)}^2\right] \,\dd t,
\] 
where $\abs{\oma}$ denotes the measure of $\oma$.
Since there is a constant depending only on the shape-regularity of the elements of the patch $\oma$ such that $h_{\oma} \abs{\oma}^{1/2}\norm{\nabla \psia}_{\oma}\lesssim \norm{\psia}_{L^1(\oma)}$, we find that $\abs{R_3}^2 +\abs{R_4}^2 \lesssim \int_{I_n} \norm{\nabla(u-u_{\th})}_{\oma}^2+\norm{f-\Pia f}_{H^{-1}(\oma)}^2\dd t$.
Finally, it is straightforward to show that  $\abs{R_5}^2 \leq \int_{I_n} \norm{ f -\Pia f}_{H^{-1}(\oma)}^2\,\dd t $. Therefore, the above bounds on the quantities $R_i$ imply \eqref{eq:X_norm_local_efficiency_2} for the case where $\ver\in\calVhint$ is an interior vertex.

The analogous result for the case where $\ver\in\calVhext$ is a boundary vertex poses fewer difficulties than the case of interior vertices, owing to the fact that $u-\Uth \in \Hstar$ for a.e. $t\in I_n$, since $u$ and $\Uth$ are both in $X$ and therefore have vanishing trace on $\p\oma\cap \DO$.

Using the triangle inequality $\norm{\nabla(u-\Uth)}_{\oma} \leq \norm{\nabla(u-u_{\th})}_{\oma} + \norm{\nabla(u_{\th}-\Uth)}_{\oma}$, it is then straightforward to obtain \eqref{eq:Xnorm_local_efficiency} and \eqref{eq:Xnorm_global_efficiency} from \eqref{eq:X_norm_local_efficiency_1} and \eqref{eq:X_norm_local_efficiency_2}.\qed

\ifIMA
\section*{Funding}
This project has received funding from the European Research Council (ERC) under the European Union's Horizon 2020 research and innovation program (grant agreement No 647134 GATIPOR).
\fi

\ifIMA
\input{heat_IMA_biblio}
\else

\fi

\end{document}